\documentclass[11pt]{article}
\usepackage[margin=1in]{geometry}
\usepackage{lmodern,microtype,amsmath,amssymb,amsthm,mathtools}
\usepackage[hidelinks]{hyperref}
\usepackage{enumitem}
\newtheorem{theorem}{Theorem}[section]
\newtheorem{lemma}[theorem]{Lemma}
\newtheorem{proposition}[theorem]{Proposition}\newtheorem{conjecture}[theorem]{Conjecture}
\newtheorem{remark}[theorem]{Remark}
\newcommand{\C}{\mathcal C}
\newcommand{\E}{\mathbb E}

\newcommand{\R}{\mathbb R}

\title{Sharp $L^6$ Fourier moments  and (3,$\delta$)-energy estimates for fractal measures on the truncated cone in $\R^3$}
\author{Ciprian Demeter}
\date{}
\begin{document}
	\maketitle
	\begingroup
	\renewcommand{\thefootnote}{}
	\footnotetext{
		The author is partially supported by NSF grant DMS-2349828.}
	\endgroup
	\begin{abstract}
		For a $\delta$-separated $(\delta,d)$-Katz--Tao set $P$ on the truncated cone in $\R^3$, $1<d\le 2$, we prove the sharp energy bound
		\[
		\E_{3,\delta}(P)\lesssim_{\varepsilon} \delta^{-\varepsilon}
		\max\{\delta^{-(3d+2)},\delta^{-(5d-1)}\}.
		\]
		The proof applies the Fu--Ren incidence theorem  to triples with a fixed sum.  As a corollary we also deduce a sharp $L^6$ moment estimate for the Fourier transform of Frostman measures supported on the cone.
	\end{abstract}
	
	\section{Introduction}
	
	\subsection{The main results for the cone}
	
	We consider the truncated cone
	\[
	\C=\{(r\cos\theta,r\sin\theta,r):\tfrac12\le r\le1\},\]
	and the key notation associated with it
	\[Q(x,y,t)=t^2-x^2-y^2,\qquad e(\theta)=(\cos\theta,\sin\theta).
	\]
	
	A probability measure $\mu$ on $\C$ is called $d$-Frostman if there is a constant $C$ such that
	$\mu(B(x,r))\le Cr^d$ for all $0<r\le2$.
	We use $\widehat\mu(\xi)=\int e^{-2\pi i x\cdot\xi}\,d\mu(x)$.

	\begin{theorem}\label{thm:fourierf}
		Let $1<d\le2$, and let $\mu$ be a probability measure supported on $\C$ such that
		\begin{equation}\label{frostmanf}
			\mu(B(x,r))\le C r^d,\qquad 0<r\le 2.
		\end{equation}
		For every $\varepsilon>0$ and $R\ge2$,
		\begin{equation}\label{fourierboundf}
			\int_{B(0,R)}|\widehat\mu(\xi)|^6\,d\xi
			\lesssim_{\varepsilon,C} R^{\varepsilon}
			\max\{R^{5-3d},R^{2-d}\}.
		\end{equation}
	\end{theorem}
	When $d\le 1$ the supporting cone loses its relevance, as the measure could live entirely on a line segment.

	The exponents $5-3d,2-d$ are both attainable, as explained in Remark \ref{djhfdhufhrgtgu9thu90}.

	In the last section we derive Theorem \ref{thm:fourierf} from a discrete energy theorem, whose proof occupies most of this paper.

	We call a subset $P$ of a metric space  a $(\delta,d,C)$-KT set if (and we will call this a KT non-concentration condition; KT is short for Katz--Tao who first showed the relevance of such sets for continuous problems)
	\begin{equation}\label{kt}
		\#(P\cap B(x,\rho))\le C(\rho/\delta)^d,
		\qquad \delta\le\rho\le 2.
	\end{equation}
	Let $P\subset\C$ be a finite $\delta$-separated set, $N=\#P$.  Condition \eqref{kt} gives $N\lesssim_C\delta^{-d}$.
	We  define the $(3,\delta)$-energy (in previous literature this was sometimes called (6,$\delta$)-energy) by
	\[
	\E_{3,\delta}(P)=\#\{(p_1,\dots,p_6)\in P^6:
	|p_1+p_2+p_3-p_4-p_5-p_6|\le C_*\delta\}.
	\]

	When $1<d\le2$, we will prove the sharp bound
	\begin{equation}\label{goal}
		\E_{3,\delta}(P)\lesssim_{\varepsilon,C,C_*}
		\delta^{-\varepsilon}\max\{\delta^{-(3d+2)},\delta^{-(5d-1)}\}.
	\end{equation}
	Both exponents occur for sets with $\#P\asymp\delta^{-d}$, as we will prove below.

	The case $0<d\le1$ does not have any substance. The trivial upper bound $O(N^5)$ for the energy is easily attainable by placing the points on a line inside the cone,  in arithmetic progression. This is precisely possible because the Hausdorff dimension of the line is at least as big as the ``dimension" $d$ of the points.

	From here on we assume that $1<d\le2$.
	\\
	\\
	\textbf{The flat example.}
	The $\delta$-neighborhood of the truncated cone is nearly planar on an angular sector of width $\sim \sqrt\delta$. Fix $\theta$ and  let $e=e(\theta)$ and $e^\perp=e'(\theta)$. Points of the form
	\begin{equation}\label{flatparam}
		p(r,y)=\big(re+ye^\perp,\sqrt{r^2+y^2}\big),
		\qquad \tfrac12\le r\le\tfrac34,\quad 0\le y\le c\sqrt\delta
	\end{equation}
	lie on a fixed such sector of $\C$. Note that
	\begin{equation}\label{flatapprox}
		p(r,y)=(re+ye^\perp,r)+O(c^2\delta).
	\end{equation}
	Take $r$ in a $\delta$-progression and $y$ in a $\delta^{d/2}$-progression. Each grid has
	\[
	n\asymp \delta^{-1}\frac{\sqrt\delta}{\delta^{d/2}}
	=\delta^{-(d+1)/2}
	\]
	points. The planar grid in \eqref{flatapprox} has $\gtrsim n^5$ additive 6-tuples: independently choose five points in the grid, and a positive proportion of the choices determine a sixth point within the grid. By \eqref{flatapprox}, these are $O(\delta)$-additive 6-tuples on the cone.
	
	We select $J\asymp\delta^{-(d-1)/2}$ angular sectors (of width $\sim \sqrt{\delta}$) with separation $\asymp\delta^{(d-1)/2}$. For each sector use the grid above, and we let $P$ be their union. Then
	\[
	\#P\asymp Jn\asymp\delta^{-d},\qquad
	\E_{3,\delta}(P)\gtrsim Jn^5=\delta^{-(3d+2)}.
	\]
	The Katz--Tao condition follows by counting grid points in a ball. For $\rho\le\sqrt\delta$ a ball meets $O(1)$ grids and has $O((\rho/\delta)(1+\rho/\delta^{d/2}))$ points. For $\rho\ge\sqrt\delta$ it meets $O(\rho/\delta^{(d-1)/2})$ grids and has $O(\rho\delta^{-(d+1)/2})$ points in each. Both bounds are $O((\rho/\delta)^d)$, as needed.
	\\
	\\
	\textbf{The radial arithmetic progression example.}
	Let
	\[
	h=\delta^{d-1},\qquad
	R_h=\{\tfrac12+jh:0\le j\lesssim h^{-1}\},
	\]
	and let $\Theta$ be a $\delta$-separated subset of $[0,\frac\pi2]$, with $\#\Theta\asymp\delta^{-1}$. Define
	\[
	P=\{(r\cos\theta,r\sin\theta,r):r\in R_h,\ \theta\in\Theta\}.
	\]
	Then $\#P\asymp h^{-1}\delta^{-1}=\delta^{-d}$. Moreover, for $\delta\le\rho\le1$,
	\[
	\#(P\cap B(x,\rho))\lesssim
	(1+\rho/h)\rho/\delta\lesssim(\rho/\delta)^d.
	\]
	The height coordinate of every triple sum belongs to $R_h+R_h+R_h$, which contains only $O(h^{-1})$ elements. At each such height, the two remaining sum coordinates occupy $O(\delta^{-2})$ squares of side $O(\delta)$ in the plane. Thus the $N^3\asymp\delta^{-3d}$ triple sums occupy at most $O(h^{-1}\delta^{-2})=O(\delta^{-(d+1)})$ $\delta$-cubes in $\R^3$. Cauchy--Schwarz gives
	\begin{equation}\label{aplower}
		\E_{3,\delta}(P)\gtrsim
		\frac{N^6}{\delta^{-(d+1)}}
		\asymp \delta^{-(5d-1)}.
	\end{equation}

	The two lower bounds cross at $d=3/2$. For $1<d\le3/2$ the flat term dominates, while for $3/2\le d\le2$ the radial term dominates. Our main theorem gives the corresponding upper bound.
	In fact, as a by-product of our method, a more refined bound emerges, one that takes into account the size $N$ of the set.
	\begin{theorem}\label{thm:main}
		Let $1<d\le2$ and let $P\subset\C$ be a $\delta$-separated $(\delta,d,C)$-KT set, with $N=\#P$. For every $\varepsilon>0$,
		\begin{equation}\label{mainbound}
			\E_{3,\delta}(P)\lesssim_{\varepsilon,C,C_*}
			N\delta^{-\varepsilon}\max\{\delta^{-(2d+2)},\delta^{-(4d-1)}\}.
		\end{equation}
		In particular, \eqref{goal} holds and is sharp up to $\delta^{-\varepsilon}$ for sets of maximal cardinality.
	\end{theorem}
	
	\subsection{The conjectures for the parabolic cylinder and the parabola}
	Consider the truncated parabolic cylinder
	\[
	\mathcal Y=\{(x,x^2,z):x\in[0,1],\ z\in[0,1]\}\subset\mathbb R^3.
	\]
	The cylinder is often a good model for the cone, as  both surfaces have zero Gaussian curvature. The analogous conjecture we make for a $(\delta,d)$-KT set $S\subset\mathcal Y$, $1<d\le2$, is
	\begin{equation}\label{cylinderconj}
		\E_{3,\delta}(S)\lesssim_\varepsilon
		\delta^{-\varepsilon}\max\{\delta^{-(3d+2)},\delta^{-(5d-1)}\}.
	\end{equation}
	The two exponents are the same as for the cone, and their justification is also similar. The first obstruction comes from $\sqrt{\delta}$-wide strips of the cylinder, which are nearly planar at scale $\delta$. For the second, we use a set which projects to an arithmetic progression in the $x$-coordinate.  Both constructions yield the same exponents as above, via the same type of counting.
	
	The cylinder conjecture is at least as hard as the corresponding 3-energy conjecture for the parabola, introduced in  \cite{Orponen}.

	\begin{conjecture}[3-energy on the parabola]\label{conj:parabola}
		Let $0<s\le1$, and let $A\subset\{(x,x^2):0\le x\le1\}$ be a
		$\delta$-separated $(\delta,s,C)$-KT set. Then for every $\varepsilon>0$,
		\begin{equation}\label{parabconj}
			\E_{3,\delta}(A)\lesssim_{\varepsilon,C,C_*}
			\delta^{-\varepsilon}\max\{\delta^{-3s},\delta^{-(5s-1)}\}.
		\end{equation}
	\end{conjecture}

	Indeed, let $A\subset\{(x,x^2):0\le x\le1\}$ be a maximal $(\delta,s)$-KT set, where $s=d-1$, and let $Z$ be a $\delta$-arithmetic progression of length comparable to $\delta^{-1}$. Then $S=A\times Z$, with the coordinates ordered as above, is a $(\delta,s+1)$-KT set. Each approximate additive 6-tuple in $A$ can be paired with $\gtrsim\delta^{-5}$ exact additive 6-tuples in $Z$. Thus
	\[
	\E_{3,\delta}(S)\gtrsim\delta^{-5}\E_{3,c\delta}(A)
	\]
	for a fixed  $c=O(1)$. Applying \eqref{cylinderconj} with $d=s+1$ would give
	\begin{equation}\label{parabconj}
		\E_{3,\delta}(A)\lesssim_\varepsilon
		\delta^{-\varepsilon}\max\{\delta^{-3s},\delta^{-(5s-1)}\},\qquad 0<s\le1,
	\end{equation}
	which is precisely the conjectured parabola numerology. This gives a concrete sense in which the cylinder problem is harder. We do not know a comparable implication from the cone estimate to the parabola estimate.
	
	Orponen \cite{Orponen} initiated the study of the $\delta$-energy problem for the parabola, deriving a power improvement over the (comparatively speaking, elementary) fourth-moment estimate, using Furstenberg-set theory. Subsequent progress is described in \cite{DO}, where O'Regan and the author prove the bound $\E_{3,\delta}(A)\lesssim_\varepsilon\delta^{-7s/2-\varepsilon}$ for $s\le2/3$. We refer to that paper for the intervening results and their references. More recently, Orponen and Shmerkin \cite[Theorem~6.9]{OS} proved the sharp triple-sumset dimension formula $\dim_{\mathrm H}(A+A+A)=3s$ for every $s$-Ahlfors--David regular subset $A$ of the parabola with $0<s\le\tfrac12$. This Hausdorff-dimension result does not by itself imply the discretized energy estimate in Conjecture~\ref{conj:parabola}.
	
	There is a contrast with familiar Fourier-analytic estimates. The cone and parabolic cylinder share similar critical exponents for a multitude of problems: they have the same critical $\ell^2L^6$ decoupling exponent, and comparable critical restriction and $L^4$ square-function phenomena. The cone is usually (both conceptually and technically) much harder, notably in reverse square-function estimates, see \cite{GWZ} and \cite{CDW}, because its curved and flat directions interact under rescaling. In our  context however,  the cone energy admits a full resolution, whereas the cylinder conjecture contains the unsolved parabola energy problem. 
	
	\begin{remark}[Energy and the ring problem]
		For the parabola, Orponen's improvement over the elementary energy exponent is connected to the discretized ring problem and to Furstenberg-set estimates. This is one motivation for the Furstenberg-theoretic methods used in \cite{DO}. We do not know an analogous implication from improved energy on the cone to a ring or sum--product theorem. The radial variable and the angular variable do not combine into the same one-dimensional additive--multiplicative structure. Our proof below uses planar incidences of lower complexity than the Furstenberg-set theorems known to have non-trivial sum-product consequences.
	\end{remark}
	
	\section{A warm-up: exact energy ($\delta=0$) and Szemer\'edi--Trotter}
	We explain how the argument works in the technically simpler case when $\delta=0$. Let $P\subset\C$ be finite, without any non-concentration assumptions. For each angle $\theta$, the segment
	\[
	L_\theta=\{(r\cos\theta,r\sin\theta,r):\tfrac12\le r\le1\}
	\]
	is called a \emph{generator} of the truncated cone.  Let $M$ be the largest number of points of $P$ on any generator, and write
	\[
	\E_3(P)=\#\{(p_1,\ldots,p_6)\in P^6:
	p_1+p_2+p_3=p_4+p_5+p_6\}.
	\]
	We prove the discrete analog of Theorem \ref{thm:main}.
	\begin{proposition}\label{prop:exact}
		For every finite $P\subset\C$,
		\[
		\E_3(P)\lesssim N^{13/3}+NM^4.
		\]
	\end{proposition}
	\begin{remark}
		The estimate in Proposition \ref{prop:exact} serves only as a warm-up for the transformation and -to some extent- for the incidence mechanism used below. Jing and Wu \cite{JW} recently proved that, for a finite set $P$ on a fixed irreducible algebraic surface of degree at least two (such as the cone), its fourth energy obeys
		\[
		\E_2(P)\lesssim_\varepsilon M N^{2+\varepsilon},
		\]
		where $M\ge1$ bounds the number of points on any line contained in the surface. Applying Young's convolution inequality gives
		\[
		\E_3(P)=\|1_P*1_P*1_P\|_{\ell^2}^2
		\le N^2\|1_P*1_P\|_{\ell^2}^2
		\lesssim_\varepsilon M N^{4+\varepsilon}.
		\]
		This is stronger than the $N^{13/3}$ term of Proposition \ref{prop:exact} for bounded $M$, though the bounds have different advantages when $M$ is large. We present the elementary Szemer\'edi--Trotter proof solely to introduce the nonlinear incidence transformation needed in the $\delta$-problem.
	\end{remark}

	In preparation for the proof of Proposition \ref{prop:exact}, we distinguish the two  possible cases for the triple sum of points on the cone. If
	$ p_i=r_i(e(\theta_i),1)$ with $\frac12\le r_i\le 1$ and $q=p_1+p_2+p_3$, then the definition of $Q$ gives
	\begin{align}
		Q(q)&=(r_1+r_2+r_3)^2-\left|\sum_{i=1}^3r_ie(\theta_i)\right|^2\notag\\
		&=2\sum_{1\le i<j\le3}r_ir_j\bigl(1-\cos(\theta_i-\theta_j)\bigr).\label{nullidentity}
	\end{align}
	To explain the second equal sign, expand the two squares separately:
	\begin{align*}
		(r_1+r_2+r_3)^2
		&=\sum_{i=1}^3r_i^2+2\sum_{i<j}r_ir_j,\\
		\left|\sum_{i=1}^3r_ie(\theta_i)\right|^2
		&=\sum_{i=1}^3r_i^2+2\sum_{i<j}r_ir_j\,e(\theta_i)\cdot e(\theta_j)\\
		&=\sum_{i=1}^3r_i^2+2\sum_{i<j}r_ir_j\cos(\theta_i-\theta_j).
	\end{align*}
	Here we used $|e(\theta)|=1$ and
	$e(\theta_i)\cdot e(\theta_j)=\cos(\theta_i-\theta_j)$.
	Subtracting gives \eqref{nullidentity}.
	Thus $Q(q)\ge0$. 
	
	We call a sum $q$ \emph{zero} if $Q(q)=0$ and \emph{nonzero} if $Q(q)>0$.
	Since every summand in \eqref{nullidentity} is nonnegative, a sum is zero exactly when all three
	points lie on one generator. Nonzero means that at least two of their angular directions differ.
	Only the nonzero case requires incidence geometry, while the zero case is trivial. We start with the nonzero case. 
	
	For $q\in P+P+P$ write
	\[
	n(q)=\#\{(p_1,p_2,p_3)\in P^3:p_1+p_2+p_3=q\}.
	\]
	Then $\E_3(P)=\sum_qn(q)^2$ and $\sum_qn(q)=N^3$. The distinguished feature of our argument is that we will derive uniform bounds for $n(q)$, when $Q(q)>0$. This allows us to fix $q$.
	We will apply two transformations. The first one normalizes the geometry of points, to lay the ground for uniform estimates for $n(q)$. The second one recasts triple sums contributing to $n(q)$ as planar incidence estimates between points and lines. 
	
	\subsection{A fixed sum and the linear  Lorentz rescaling}

	Write $q=(q_1,q_2,t)$. Since the cone and its truncation are invariant under rotations in the horizontal $(x,y)$-plane, we may rotate the coordinates so that $q=(s,0,t)$, where $s=\sqrt{q_1^2+q_2^2}\ge0$. This rotation preserves distances and additive relations. For a nonzero sum $Q(q)>0$, we have $t>s$. Set $h=\sqrt{t^2-s^2}$ and consider the Lorentz rescaling
	\begin{equation}\label{lorentz}
		B_q(x,y,z)=\left(\frac{tx-sz}{h},y,\frac{tz-sx}{h}\right),
		\qquad A_q=h^{-1}B_q.
	\end{equation}
	The map $A_q$ is linear and satisfies $A_qq=(0,0,1)$. Moreover, $Q(A_qp)=h^{-2}Q(p)$, so $A_q$ maps the upper cone to itself. By linearity, if $p_1+p_2+p_3=q$, then
	\[
	A_qp_1+A_qp_2+A_qp_3=(0,0,1).
	\]
	
	Apply $A_q$ to all three points. Their images still lie on the upper cone, and their sum is now $(0,0,1)$. We may therefore write
	\[
	A_qp_i=r_i'(e(\theta_i'),1).
	\]
	The primed radii and angles generally differ from the original ones; in particular, the transformed radii need not belong to $[1/2,1]$. 
	For points in a triple with sum $q$, however, they satisfy absolute bounds. Indeed, writing the original points as $p_i=r_i(e(\theta_i),1)$, the formula for $A_q$ gives
	\[
	r_i'=\frac{r_i(t-s\cos\theta_i)}{h^2}
	=\frac{r_i}{t+s}+\frac{r_is(1-\cos\theta_i)}{h^2}
	\ge \frac{r_i}{t+s}\ge\frac1{12}.
	\]
	Here we used $r_i\ge1/2$, $t=r_1+r_2+r_3\le3$, and $0\le s<t$, hence $t+s\le6$. On the other hand, the normalized sum equation gives
	\[
	r_1'+r_2'+r_3'=1,
	\qquad \sum_{i=1}^3r_i'e(\theta_i')=0.
	\]
	Thus, for each $i$, the triangle inequality implies
	\[
	r_i'=\left|\sum_{j\ne i}r_j'e(\theta_j')\right|
	\le\sum_{j\ne i}r_j'=1-r_i',
	\]
	so $r_i'\le1/2$. We have proved that every point in a contributing triple to $q$ has
	\[
	\frac1{12}\le r_i'\le\frac12.
	\]
	\begin{remark}
		\label{hdufherugy8t9ug90tru}
		We record a more refined observation. This is not needed in the discrete argument presented in this section, but will be needed when we deal with $\delta$-approximate incidences.
		
		If the transformed radii are ordered as $r_{(1)}'\ge r_{(2)}'\ge r_{(3)}'$, then
		$2r_{(2)}'+1/12\le1$. In particular, at least two of the three transformed radii belong to
		$[1/12,11/24]\subset[1/24,1/2-1/24]$.
	\end{remark}

	To simplify notation, from now on we omit the primes and write $p_i=r_i(e(\theta_i),1)$ for the transformed points.
	
	\subsection{The nonlinear change of coordinates}
	
	At this point, we could recast  a triple contributing to $q$ as an incidence between points and ellipses having two degrees of freedom. The desired estimate would be a consequence of the generalized Szemer\'edi--Trotter theorem due to Pach and Sharir, \cite{PS}. We instead follow an alternative approach. The role of the nonlinear map $V$ is to straighten ellipses into lines. This perspective will be useful later when we estimate the $\delta$-energy, and apply an incidence theorem available for straight tubes.

	For the transformed sum $q=(0,0,1)$, direct calculation gives
	\begin{align}
		Q(q-p_1-p_2)
		&=1-2(r_1+r_2)+2r_1r_2(1-\cos(\theta_1-\theta_2))\label{compat}\\
		&=4r_1r_2\left[\left(\frac1{2r_1}-1\right)
		\left(\frac1{2r_2}-1\right)-\cos^2\frac{\theta_1-\theta_2}{2}\right].\nonumber
	\end{align}
	For $0<r<1/2$, choose a continuous angular coordinate $\theta$ on a fixed interval of length less than $2\pi$ and define the non-linear 
	\begin{equation}\label{Vmap}
		V(p)=V(r,\theta)=\sqrt{\frac{2r}{1-2r}}\ e(\theta/2),
		\qquad p=r(e(\theta),1).
	\end{equation}
	The notation $V(p)$ means that the two parameters $r$ and $\theta$ are the radial and angular coordinates of the cone point $p$; thus $V(p_i)=V(r_i,\theta_i)$. The angle is defined modulo $2\pi$, but $e(\theta/2)$ changes sign when $\theta$ is increased by $2\pi$. We therefore use finitely many angular intervals on which $\theta$ is chosen continuously, making $V$ single-valued.

	Since \(r_1+r_2+r_3=1\) and all three radii are positive, at most one of them can be greater than or equal to (or better said, equal to) \(1/2\). Thus every triple has at least two points whose transformed radii belong to \((0,1/2)\), where \(V\) is well-defined. We apply \(V\) only to these two points, which we choose to be $p_1$, $p_2$,  and split the count according to the three possible choices of pair. The third point is uniquely determined by the sum equation. Later, in Section 3.3, we will take the two selected radii in the compact sub-interval \([1/48,1/2-1/48]\), which is needed for the quantitative distortion estimates.
	
	We next turn the equality 
	$$p_1+p_2+p_3=q=(0,0,1)$$
	into an incidence between a point (namely $V(p_2)$) and a line parametrized by $V(p_1)$. The role of $p_3$ will be solely coming from the fact that it belongs to the cone. The fact that it is in $P$ is irrelevant.  
	
	Set
	\[
	a_i=\frac{1-2r_i}{2r_i}>0,\qquad u_i=V(p_i)
	=a_i^{-1/2}e(\theta_i/2),\qquad i=1,2.
	\]
	Using $e(\alpha)\cdot e(\beta)=\cos(\alpha-\beta)$, we have
	\[
	u_1\cdot u_2=\frac{\cos((\theta_1-\theta_2)/2)}{\sqrt{a_1a_2}},
	\qquad
	\cos^2\frac{\theta_1-\theta_2}{2}
	=a_1a_2(u_1\cdot u_2)^2.
	\]
	Substituting into \eqref{compat} gives
	\[
	\begin{aligned}
		Q(q-p_1-p_2)
		&=4r_1r_2a_1a_2\bigl(1-(u_1\cdot u_2)^2\bigr)\\
		&=(1-2r_1)(1-2r_2)\bigl(1-(V(p_1)\cdot V(p_2))^2\bigr).
	\end{aligned}
	\]
	Since $0<r_1,r_2<1/2$, the prefactor is nonzero. Hence
	$Q(q-p_1-p_2)=0$ if and only if
	\begin{equation}\label{linecondition}
		V(p_1)\cdot V(p_2)=\pm1.
	\end{equation}
	Fix $p_1$ and write $V(p_1)=(a,b)$ and $V(p_2)=(x,y)$. Since
	$V(p_1)\ne0$, equation \eqref{linecondition} says
	\[
	ax+by=1\qquad\text{or}\qquad ax+by=-1.
	\]
	Each equation describes a line in the $(x,y)$-plane. Thus, for a fixed choice
	of sign, $p_1$ determines a line $\ell_{p_1}$, and every compatible $p_2$
	gives an incidence $V(p_2)\in\ell_{p_1}$.
	
	We check that no multiplicities are introduced. First, the map
	$p\mapsto V(p)$ is injective. Indeed,
	\[
	|V(p)|^2=\frac{2r}{1-2r},
	\]
	and the right-hand side is strictly increasing for $0<r<1/2$.
	Hence $V(p)$ determines $r$. Its direction determines $\theta/2$, and
	therefore $\theta$, once a continuous choice of angle has been fixed.
	Second, for a fixed sign the line $ax+by=\pm1$ determines $(a,b)$
	uniquely: the right-hand side is fixed and nonzero. Hence the map from
	$p_1$ to its line is injective as well. We count the two signs separately,
	so that writing $u\cdot v=1$ and $(-u)\cdot v=-1$ does not identify two
	lines within the same family.

	\subsection{Nonzero sums}
	\begin{lemma}\label{lem:exacttime}
		For every $q$ with $Q(q)>0$, $n(q)\lesssim N^{4/3}$.
	\end{lemma}
	\begin{proof}
		We recall the reduction from the previous two subsections. Since $Q(q)>0$, after rotating the horizontal coordinates we may write $q=(s,0,t)$ with $h=\sqrt{t^2-s^2}>0$. Thus the Lorentz map $A_q$ is well-defined, maps the upper cone to itself, and satisfies $A_qq=(0,0,1)$. By linearity, a triple with $p_1+p_2+p_3=q$ satisfies $A_qp_1+A_qp_2+A_qp_3=(0,0,1)$. Writing the transformed points as $r_i(e(\theta_i),1)$, comparison of the third coordinates gives $r_1+r_2+r_3=1$. In particular, at least two radii are strictly smaller than $1/2$. This is needed in order to make the use of $V$ possible.  Split according to the choice of the pair. For each such pair, \eqref{linecondition} gives incidences between at most $N$ distinct points and at most $N$ distinct lines. Szemer\'edi--Trotter gives $O(N^{4/3}+N)$ incidences. Once the pair is fixed, the third point is uniquely determined. There are only three choices of pair and finitely many charts.
	\end{proof}
	It follows immediately that
	\begin{equation}\label{extime}
		\sum_{Q(q)>0}n(q)^2\le(\sup_{Q(q)>0}n(q))\sum_qn(q)
		\lesssim N^{13/3}.
	\end{equation}
	
	\subsection{Zero sums}
	\begin{lemma}\label{lem:exactnull}
		The contribution to $\E_3(P)$ from $Q(q)=0$ is $O(NM^4)$.
	\end{lemma}
	\begin{proof}
		By \eqref{nullidentity}, $Q(q)=0$ forces all three points of a contributing triple to lie on the same generator. Two triples with the same zero sum lie on the same generator as each other. If a generator contains $m$ points, its energy is at most $m^5$ because five points determine the sixth. Summing over generators gives $\sum m^5\le M^4\sum m=NM^4$.
	\end{proof}
	\begin{proof}[Proof of Proposition \ref{prop:exact}]
		Combine \eqref{extime} and Lemma \ref{lem:exactnull}.
	\end{proof}

	\begin{remark}[Exact and thickened energy]
		For a $\delta$-separated maximal-size KT set one always has $M\lesssim\delta^{-1}$, and the exact theorem gives $\E_3(P)\lesssim\delta^{-13d/3}+\delta^{-(d+4)}$. This is not the sharp thickened exponent. In the latter problem a $\sqrt\delta$-wide plate, not merely a generator, behaves as a plane at tolerance $O(\delta)$. Ordinary Szemer\'edi--Trotter is therefore not the needed incidence input.
	\end{remark}

	\section{The  $\delta$-energy estimate}
	This section is devoted to proving Theorem \ref{thm:main}. All implicit constants  may depend on $C$ and $C_*$.
	Partition the bounded region containing $P+P+P$ into cubes of side $c\delta$, with $c>0$ fixed. Let $n(q)$ be the number of triples with sum in the cube $q$. We may write
	\begin{equation}\label{energycubes}
		\E_{3,\delta}(P)\lesssim\sum_{q}n(q)^2,\qquad
		\sum_qn(q)=N^3.
	\end{equation}
	For a triple $(p_1,p_2,p_3)$, its \emph{angular diameter} is
	\[
	\operatorname{diam}_{\rm ang}(p_1,p_2,p_3)
	:=\max_{1\le i<j\le3}\operatorname{dist}_{\mathbb S^1}
	(\theta_i,\theta_j),
	\]
	where $\operatorname{dist}_{\mathbb S^1}$ is the shorter angular distance.
	We say that the triple has angular diameter $\sigma$ when this number lies
	between  $\sigma/2$ and $\sigma$. Identity
	\eqref{nullidentity}, together with $1/2\le r_i\le1$, shows that
	$Q(p_1+p_2+p_3)\asymp\sigma^2$: the largest angular separation contributes
	$\sim\sigma^2$, while the other two contribute $O(\sigma^2)$.
	
	Assume now $\sigma^2\gg\delta$. For a sum cube  $q$ and for a triple satisfying
	$$p_1+p_2+p_3=q+O(\delta), \;\operatorname{diam}_{\rm ang}(p_1,p_2,p_3)\asymp \sigma$$
	we may still conclude that 
	$$Q(q)\asymp \sigma^2,$$
	since $Q(q)=Q(p_1+p_2+p_3)+O(\delta)$.
	The scales below $C_0\sqrt\delta$, with $C_0$ a sufficiently large
	fixed constant, are included in the near-zero contribution of
	Lemma \ref{lem:nearnull}: their sums satisfy $|Q(q)|=O(\delta)$.
	We apply Lorentz normalization only when
	$\sigma\ge C_0\sqrt\delta$. We work dyadically
	in this remaining range $C_0\sqrt\delta\le\sigma\lesssim1$.
	
	\begin{lemma}[The sector determined by a sum cube]\label{lem:qsector}
		Fix $\sqrt\delta\ll\sigma\lesssim1$. There is a boundedly
		overlapping family $\mathcal T_\sigma$ of angular arcs $\tau$ of length
		$\asymp\sigma$ and a rule assigning each sum cube $q$  with $|Q(q)|\asymp\sigma^2$ to one arc
		$\tau(q)\in\mathcal T_\sigma$,
		with the following property. If
		\[
		|Q(q)|\asymp\sigma^2,\qquad
		p_1+p_2+p_3=q+O(\delta),
		\]
		then the directions of $p_1,p_2,p_3$ all lie in $\tau(q)$.
		
		When $|Q(q)|=O(\delta)$, we may use a family $\mathcal T_{\sqrt{\delta}}$ consisting of arcs $\tau$ of length $\asymp \sqrt{\delta}$.
	\end{lemma}
	\begin{proof}
		We show the argument in the case $\sqrt\delta\ll\sigma\lesssim1$. The other case follows by simple modifications. 
		Write $p_i=r_i(e(\theta_i),1)$, $R=\sum_i r_i\asymp1$, and
		$x=\sum_i r_i e(\theta_i)$. Thus $S:=p_1+p_2+p_3=(x,R)$.
		Since $S=q+O(\delta)$ and $\delta\lesssim\sigma^2$,
		$\,Q(S)=Q(q)+O(\delta)\lesssim\sigma^2$. The identity
		\begin{equation}\label{qsectoridentity}
			\sum_{i=1}^3 r_i\left|e(\theta_i)-\frac{x}{R}\right|^2
			=R-\frac{|x|^2}{R}=\frac{Q(S)}{R}
		\end{equation}
		and $r_i\ge1/2$ imply
		$|e(\theta_i)-x/R|\lesssim\sigma$ for each $i$.
		If $\sigma$ is sufficiently small, then
		$|x|^2=R^2-Q(S)$ gives $|x|\asymp1$, and hence
		\[
		\left|e(\theta_i)-\frac{x}{|x|}\right|\lesssim\sigma.
		\]
		Writing $v=(q_1,q_2)$, we also have $|v-x|\lesssim\delta$,
		so $|v|\asymp1$ and
		\begin{equation}\label{qsectordirection}
			\operatorname{dist}_{\mathbb S^1}
			\bigl(\theta_i,\arg v\bigr)\lesssim\sigma,
			\qquad i=1,2,3.
		\end{equation}
		It follows that there is $L\lesssim \sigma$ so that each contributing triple $\theta_1,\theta_2,\theta_3$ lies inside an $L$-arc. Cover the circle with a boundedly overlapping collection $T_\sigma$ of $3L/2$-arcs, so that each $L$-arc lies inside some $\tau\in\mathcal T_\sigma$.
		
	\end{proof}
	
	\subsection{The nearly zero contribution}
	For $\tau\in\mathcal T_\sigma$ from Lemma~\ref{lem:qsector},
	define
	\[
	\C_\tau=\{r(e(\theta),1):1/2\le r\le1,\ \theta\in\tau\},
	\qquad P_\tau^{\mathrm{cone}}=P\cap\C_\tau,\qquad N_\tau=\#P_\tau^{\mathrm{cone}}.
	\]
	Thus $P_\tau^{\mathrm{cone}}$ consists of the original cone points whose angles belong to $\tau$. 
	\begin{lemma}\label{lem:sector}
		For $\sqrt\delta\lesssim\sigma\lesssim1$,
		\begin{equation}\label{sectormass}
			\max_{\tau\in\mathcal T_\sigma}N_\tau\lesssim
			\delta^{-d}\sigma^{d-1},\qquad
			\sum_{\tau\in\mathcal T_\sigma}N_\tau\lesssim N.
		\end{equation}
	\end{lemma}
	\begin{proof}
		An angular sector of aperture $O(\sigma)$ on the truncated cone is covered by $O(\sigma^{-1})$ balls of radius $O(\sigma)$. Apply \eqref{kt} in each ball. The second assertion follows from bounded overlap.
	\end{proof}

	\begin{lemma}[Near-zero sums]\label{lem:nearnull}
		Triples whose sum cube has $|Q(q)|=O(\delta)$ contribute $O(N\delta^{-2d-2})$ to $\E_{3,\delta}(P)$.
	\end{lemma}
	\begin{proof}
		Apply Lemma~\ref{lem:qsector}.
		For each output cube $q$ with $|Q(q)|=O(\delta)$, the lemma places
		all triples counted by $n(q)$ in the same sector $\tau(q)$.
		Within this sector, five points determine the sixth up to $O(1)$
		possibilities, since $P$ is $\delta$-separated. By bounded overlap
		the contribution is at most
		\[
		\sum_{\tau\in\mathcal T_{\sqrt\delta}}N_\tau^5
		\le (\max_\tau N_\tau)^4\sum_\tau N_\tau
		\lesssim N(\delta^{-d}(\sqrt\delta)^{d-1})^4
		=N\delta^{-2d-2}.
		\]
	\end{proof}
	We next deal with the far from zero-sums. 
	\subsection{The normalization at angular scale $\sigma$ for fixed $q$}
	Consider a sum cube centered at $q$ such that $Q(q)\asymp\sigma^2\gg\delta$.
	By Lemma~\ref{lem:qsector}, all triples contributing to $n(q)$
	use points in the single sector $\tau=\tau_\sigma(q)$.
	Let $q=(q_1,q_2,t)$. Rotate the horizontal coordinates so that
	$(q_1,q_2)$ becomes $(s,0)$, where $s=(q_1^2+q_2^2)^{1/2}$.
	This rotation preserves the value of $Q(q)$. It also preserves the angular diameter, in fact
	all angular parameters change by the same constant. Consequently,
	the comparison $Q(q)\asymp\sigma^2$, obtained before rotation,
	remains valid afterward. Since $q=(s,0,t)$ and $t>0$, we have
	\[
	t^2-s^2=Q(q)\asymp\sigma^2>0,
	\qquad h:=(t^2-s^2)^{1/2}=\sqrt{Q(q)}\asymp\sigma.
	\]
	We now use $A_q$ from \eqref{lorentz}. If
	\begin{equation}
		\label{wdk jhuifyerg968y09=568}
		p_1+p_2+p_3=q+O(\delta)
	\end{equation}
	then
	\begin{equation}
		\label{ch yf gtreg9tih90u8u=}
		A_qp_1+A_qp_2+A_qp_3=(0,0,1)+O(\delta/\sigma^2).
	\end{equation}
	Indeed, the symmetric matrix representing $A_q$ has eigenvalues
	$(t+s)/h^2$, $(t-s)/h^2$, and $1/h$; the first two have
	eigenvectors $(1,0,-1)$ and $(1,0,1)$. Hence its Euclidean
	operator norm is
	\[
	\|A_q\|_{\mathrm{op}}=\sup_{|v|=1}|A_qv|
	=\frac{t+s}{h^2}\asymp\sigma^{-2}.
	\]
	Thus an additive error $O(\delta)$ becomes $O(\delta/\sigma^2)$ under $A_q$.

	As explained in Remark \ref{hdufherugy8t9ug90tru}, we may assume that the radii of at least two of the three points $A_qp_1,A_qp_2,A_qp_3$ satisfy, say (assuming the error term $O(\delta/\sigma^2)$ is small enough) 
	$$1/48\le r_i',r_j'\le1/2-1/48.$$
	Any three points on the circle belong to an arc of length at most \(4\pi/3\), as follows by considering the three consecutive angular gaps. Fix a collection of 24 open arcs \(\gamma\) of length \(3\pi/2\), with equally spaced starting points. Every arc of length \(4\pi/3\) is contained in at least one of these \(\gamma\). It follows that any triple $A_q(p_1),A_q(p_2),A_q(p_3)$ will have their polar angles $\theta_i'$ inside some $\gamma$. The only relevance is that the length of $\gamma$ is $<2\pi$, as used below.
	
	Let $\tau$ be the arc associated with $q$ via Lemma \ref{lem:qsector}. We may and will assume it is centered at the origin. Indeed, after rotating the horizontal direction of $q$, the original angles of all points in $P_\tau^{\mathrm{cone}}$
	satisfy $|\theta|\lesssim\sigma$. This is the angular hypothesis used
	in Lemma~\ref{lem:boost} below. We now bound $n(q)$ using this
	sector and the finite family of transformed-angle charts $\gamma$.
	
	\paragraph{The planar map for a fixed $\gamma$.}
	We aim to do planar incidence geometry, so we shift the discussion to the plane. 
	Recall our notation 
	$p=r(e(\theta),1)$,
	$A_qp=r'(e(\theta'),1)$.  We define the planar map
	\[
	T=T_q:(r,\theta)\longmapsto V(r',\theta')
	=\sqrt{\frac{2r'}{1-2r'}}\,e(\theta'/2).
	\]
	Here $V$ is the nonlinear map from Section~2. On the fixed chart
	we require $1/48\le r'\le 1/2-1/48$ and the \emph{transformed}
	angle $\theta'\in\gamma$. Define the planar points
	\[
	P_\tau=\{T(r,\theta):r(e(\theta),1)\in P_\tau^{\mathrm{cone}},
	\ 1/48\le r'\le 1/2-1/48,\ \theta'\in\gamma\}.
	\]
	The points of $P_\tau$ lie in $\mathbb R^2$ and have cardinality $O(N_\tau)$. They arise from the points $P^{cone}_\tau$ on the cone-sector via two transformations, whose combined action is encoded by the useful planar map $T$. They are first mapped via the linear map $A_q$, then via the nonlinear map $V$ (which has a linearizing effect, leading to lines). Since the domain of $V$ will need $\theta'$ in an interval shorter than $2\pi$, we restrict $P_\tau$ to an arc $\gamma$. This is a rather unimportant technicality.

	The next lemma is instrumental in determining the KT non-concentration of the points $P_\tau$. 
	\begin{lemma}[The map $T$]\label{lem:boost}
		Let $q=(s,0,t)$ and $h=(t^2-s^2)^{1/2}\asymp\sigma$.  The
		intermediate polar coordinates of $A_qp$ satisfy
		\begin{equation}\label{boostcoords}
			r'=\frac{r(t-s\cos\theta)}{h^2},\qquad
			\tan\frac{\theta'}2=\frac{t+s}{h}\tan\frac\theta2.
		\end{equation}
		Suppose $|\theta|\lesssim\sigma$ and
		$1/48\le r'\le1/2-1/48$.  For a Euclidean ball of radius $w$ in the
		\emph{target $V$-plane}, the preimage under $T$ is contained in a
		region of radial width $O(w)$ and angular width $O(\sigma w)$.
		All constants are uniform in $\sigma$.
	\end{lemma}
	\begin{proof}
		
		The three coordinates of $A_qp$ are
		\[
		x'=\frac{r(t\cos\theta-s)}{h^2},\qquad
		y'=\frac{r\sin\theta}{h},\qquad
		r'=\frac{r(t-s\cos\theta)}{h^2}.
		\]
		Since $r'=\sqrt{x'^2+y'^2}$, the half-angle identity gives
		\[
		\tan\frac{\theta'}2=\frac{y'}{r'+x'}
		=\frac{h\sin\theta}{(t-s)(1+\cos\theta)}
		=\frac{t+s}{h}\tan\frac\theta2.
		\]
		Put $k=(t+s)/h\asymp\sigma^{-1}$.  For
		$|\theta|\lesssim\sigma$,
		\[
		\frac{d\theta'}{d\theta}
		=\frac{k(1+\tan^2(\theta/2))}{1+k^2\tan^2(\theta/2)}
		\asymp\sigma^{-1},\qquad
		\left|\frac{d\theta}{d\theta'}\right|\lesssim\sigma.
		\]
		Also $t-s\cos\theta=(t-s)+s(1-\cos\theta)\asymp\sigma^2$ and
		\[
		r=\frac{h^2r'}{t-s\cos\theta}.
		\]
		Differentiating shows $|\partial r/\partial r'|\lesssim1$ and
		$|\partial r/\partial\theta'|\lesssim1$: indeed,
		the $\theta$-derivative of the coefficient is
		$O(\sigma^{-1})$, which is multiplied by
		$d\theta/d\theta'=O(\sigma)$.
		
		It remains to take into account  $V$. On $1/48\le r'\le1/2-1/48$,
		\[
		V(r',\theta')=f(r')e(\theta'/2),\qquad
		f(r')=\sqrt{\frac{2r'}{1-2r'}}.
		\]
		Indeed,
		\[
		f'(r')=\frac{1}{\sqrt{2r'}(1-2r')^{3/2}},
		\]
		so both $f$ and $f'$ are bounded above and below by positive constants
		on this fixed radial interval. To see explicitly why $V$ has a
		uniformly bi-Lipschitz inverse on its image, take two points
		$u_j=V(r'_j,\theta'_j)=\rho_j e(\theta'_j/2)$, where
		$\rho_j=f(r'_j)$ and $\theta'_j\in\gamma$. Then
		\[
		\begin{aligned}
			|u_1-u_2|^2
			&= (\rho_1-\rho_2)^2
			+\rho_1\rho_2|e(\theta'_1/2)-e(\theta'_2/2)|^2\\
			&= (\rho_1-\rho_2)^2
			+4\rho_1\rho_2\sin^2\!\left(\frac{\theta'_1-\theta'_2}{4}\right)\\
			&\asymp |r'_1-r'_2|^2+|\theta'_1-\theta'_2|^2.
		\end{aligned}
		\]
		For the last comparison, the mean value theorem controls
		$|\rho_1-\rho_2|$ by $|r'_1-r'_2|$ in both directions,
		while $|\theta'_1-\theta'_2|\le 3\pi/2$ ensures that
		$|\sin((\theta'_1-\theta'_2)/4)|\asymp|\theta'_1-\theta'_2|$.
		The constants are independent of $q$ and $\sigma$.
		Equivalently, the inverse can be written as
		\[
		r'=\frac{|u|^2}{2(1+|u|^2)},\qquad
		\theta'=2\operatorname{Arg}_{\gamma/2}(u),
		\]
		where $\operatorname{Arg}_{\gamma/2}$ denotes the branch of the
		argument taking values in the lifted interval $\gamma/2$.
		This explains the role of restricting the transformed angle to an
		arc $\gamma$ of length less than $2\pi$: it allows a consistent
		single-valued angular branch (indeed,
		$V(r',\theta'+2\pi)=-V(r',\theta')$) and gives the uniform
		chord--angle comparison above. The finite collection of such charts
		ensures that no contributing triple is lost.
		Consequently, the inverse image of a $w$-ball under $V$ has
		$r'$-width and $\theta'$-width $O(w)$. The preceding derivative
		bounds for the inverse Lorentz coordinate map then imply radial
		width $O(w)$ and angular width $O(\sigma w)$ in the original
		coordinates.  
	\end{proof}
	
	\paragraph{The transformed incidence scale.}
	Recall that
	\[
	\|A_q\|_{\mathrm{op}}=\sup_{|v|=1}|A_qv|
	=\frac{t+s}{h^2}\asymp\sigma^{-2}.
	\]
	Thus an additive error $O(\delta)$ becomes $O(\delta/\sigma^2)$.
	More precisely, $Q(A_qv)=h^{-2}Q(v)$ and $A_qq=(0,0,1)$, so
	\[
	Q\bigl((0,0,1)-A_qp_i-A_qp_j\bigr)
	=h^{-2}Q(q-p_i-p_j)=O(\delta/\sigma^2).
	\]
	For each $a\in P_\tau$ define
	\[
	\ell_a^\pm=\{v\in\mathbb R^2:a\cdot v=\pm1\}\]
	and let
	\[\mathcal L_\tau^\pm=\{\ell_a^\pm:a\in P_\tau\}.
	\]
	The identity \eqref{linecondition} relating
	$Q((0,0,1)-A_qp_i-A_qp_j)$ to
	$V(A_qp_i)\cdot V(A_qp_j)$ shows that a compatible selected pair
	gives a point of $P_\tau$ within $O(\delta/\sigma^2)$ of a line
	in $\mathcal L_\tau^+$ or $\mathcal L_\tau^-$.  Set
	\[
	\eta:=\delta/\sigma^2.
	\]
	
	\begin{lemma}[Rescaled Katz--Tao constant]\label{lem:rescaleKT}Let $\tau\in \mathcal T_\sigma$, $\sigma\gg \sqrt{\delta}$.
		The planar point collection $P_\tau$
		and the line collections $\mathcal L_\tau^\pm$ are $(\eta,d,K_\sigma)$-KT sets, where
		\begin{equation}\label{Ksigma}
			K_\sigma\lesssim\sigma^{-d-1}.
		\end{equation}
	\end{lemma}
	\begin{proof}
		By Lemma~\ref{lem:boost}, the preimage under $T$ of a
		$w$-ball with $w\ge\eta$ has  radial
		width $O(w)$ and angular width $O(\sigma w)$.  It can be
		covered by $O(\sigma^{-1})$   balls of radius
		$O(\sigma w)$.  Since $\sigma w\ge\sigma\eta
		=\delta/\sigma\ge\delta$, the KT hypothesis for the original points on the cone (which translates into a similar condition in the $(r,\theta)$ variables) gives
		\[
		\#\{a\in P_\tau:a\in B(v,w)\}
		\lesssim \sigma^{-1}(\sigma w/\delta)^d
		=\sigma^{-d-1}(w/\eta)^d.
		\]

		The map $u\mapsto\{v:u\cdot v=\pm1\}$ is bi-Lipschitz  whenever $|u|$ is bounded above
		and away from zero.  Both requirements hold,
		since $u=V(r',\theta')$ and $1/48\le r'\le1/2-1/48$.
		Therefore the same KT bound transfers from $P_\tau$ to
		$\mathcal L_\tau^\pm$.
	\end{proof}
	
	\subsection{The incidence bound for one sum cube}
	We introduce the replacement for the Szemer\'edi--Trotter theorem. The Fu-Ren result takes advantage of the key features of our point-tube incidences derived in Lemma \ref{lem:rescaleKT}: non-concentration of both points and tubes. 
	\begin{theorem}[Fu--Ren \cite{FR}, Theorem 1.5]\label{thm:FR}
		Let $1<d\le2$ and $0<\eta<1$. Let $\mathcal P$ be a collection of
		points and $\mathcal L$ a collection of $\eta$-neighborhoods of
		unit line segments in the plane. Assume that the
		point collection and the line-parameter collection satisfy the
		$(\eta,d)$ Katz--Tao bounds with constants $K_P$ and $K_L$,
		respectively. Then, for every
		$\varepsilon>0$,
		\begin{equation}\label{FR}
			I(\mathcal P,\mathcal L)\lesssim_\varepsilon
			\eta^{-c-\varepsilon}(K_PK_L)^c
			(\#\mathcal P)^{1-c}(\#\mathcal L)^{1-c},
			\qquad c=\frac1{\max\{2d-1,2\}}.
		\end{equation}
		Here $I$ counts intersecting point--tube pairs.
	\end{theorem}
	
	We may now establish the upper bound for $n(q)$.
	
	\begin{lemma}[Fixed-sum fibre]\label{lem:fibre}
		For a sum cube $q$ with $Q(q)\asymp\sigma^2\gg\delta$,
		let $\tau=\tau(q)$ be its assigned sector from
		Lemma~\ref{lem:qsector}. Then, letting $N_\tau=\#P_\tau^{\mathrm{cone}}$
		\begin{equation}\label{fibre}
			n(q)\lesssim_\varepsilon
			\delta^{-c-\varepsilon}\sigma^{-2cd}N_\tau^{2-2c},
			\qquad c=\frac1{\max\{2d-1,2\}}.
		\end{equation}
	\end{lemma}
	\begin{proof}
		By Lemma~\ref{lem:qsector}, all triples counted by $n(q)$ lie
		in $P_\tau^{\mathrm{cone}}$, and the angular hypothesis of
		Lemma~\ref{lem:boost} holds on this sector after rotation.
		Normalize $q$ by $A_q$. Every normalized triple satisfies $r_1'+r_2'+r_3'=1+O(\eta)$. The radial observation preceding the definition of $T$, based on Remark~\ref{hdufherugy8t9ug90tru}, gives a uniform positive lower bound on all three radii and places at least two in $[1/48,1/2-1/48]$. We split into the three possible choices of this pair.
		
		For the selected pair, the thickened condition $Q(q-p_i-p_j)=O(\delta)$ becomes $Q((0,0,1)-A_qp_i-A_qp_j)=O(\eta)$. Equations \eqref{compat}--\eqref{Vmap} then turn it into an $O(\eta)$-thickened point--line incidence, between $N_\tau$ points and $N_\tau$ lines. Lemma \ref{lem:rescaleKT} and \eqref{FR} show that the number of admissible pairs is at most
		\[
		\eta^{-c-\varepsilon}K_\sigma^{2c}N_\tau^{2-2c}
		\lesssim_\varepsilon\delta^{-c-\varepsilon}\sigma^{-2cd}N_\tau^{2-2c}.
		\]
		Once the two original points are fixed, the third lies in an original $O(\delta)$-ball, giving $O(1)$ possibilities by separation. Notice that it is the original $\delta$, not the larger normalized thickness $\eta$, that is used in this last step.
	\end{proof}
	
	\subsection{Summing over the scales}
	Recall the value of $c$ in \eqref{fibre}.
	\begin{lemma}[One angular scale $\sigma$]\label{lem:onescale}
		For each dyadic $\sqrt\delta\lesssim\sigma\lesssim1$,
		\begin{equation}\label{scaleestimate}
			\begin{aligned}
				\E_\sigma
				&:=\sum_{\substack{q\\ Q(q)\asymp\sigma^2}}n(q)^2\\
				&\lesssim_\varepsilon
				N\delta^{-c-d(4-2c)-\varepsilon}
				\sigma^{(d-1)(4-2c)-2cd}.
			\end{aligned}
		\end{equation}
	\end{lemma}
	\begin{proof}
		Recall the assignment
		$q\mapsto\tau(q)$ in Lemma~\ref{lem:qsector}, and let
		$\mathcal Q_{\tau}$ be those $q$ assigned to $\tau$.
		We start with the trivial first-moment estimate
		\[
		\sum_{q\in\mathcal Q_{\tau}}n(q)\le N_\tau^3.
		\]
		Indeed, every triple in this sum belongs to
		$(P_\tau^{\mathrm{cone}})^3$, and its sum occupies exactly one
		output cube. 
		
		We next apply Lemma~\ref{lem:fibre} 
		\[
		\sup_{q\in\mathcal Q_{\tau}}n(q)
		\lesssim_\varepsilon
		\delta^{-c-\varepsilon}\sigma^{-2cd}N_\tau^{2-2c}.
		\]
		Combining the two gives
		\[
		\begin{aligned}
			\sum_{q\in \mathcal{Q}_\tau}n(q)^2
			&\lesssim_\varepsilon
			\bigl(\sup_{q\in\mathcal Q_{\tau}}n(q)\bigr)
			\sum_{q\in\mathcal Q_{\tau}}n(q)\\
			&\lesssim_\varepsilon
			\delta^{-c-\varepsilon}\sigma^{-2cd}
			N_\tau^{2-2c}\,N_\tau^3
			=\delta^{-c-\varepsilon}\sigma^{-2cd}N_\tau^{5-2c}.
		\end{aligned}
		\]
		Summing over $\tau$ and using the bounded overlap of $\tau$  yields
		\[
		\E_\sigma\lesssim_\varepsilon
		\delta^{-c-\varepsilon}\sigma^{-2cd}
		\sum_{\tau\in\mathcal T_\sigma}N_\tau^{5-2c}.
		\]
		Using the two parts of \eqref{sectormass} separately,
		\[
		\sum_\tau N_\tau^{5-2c}
		\le(\max_\tau N_\tau)^{4-2c}\sum_\tau N_\tau
		\lesssim N(\delta^{-d}\sigma^{d-1})^{4-2c}.
		\]
		This proves \eqref{scaleestimate}.
	\end{proof}
	
	\begin{proof}[Proof of Theorem \ref{thm:main}]
		For $1<d\le3/2$, $c=1/2$ and Lemma \ref{lem:onescale} reads
		\[
		\E_\sigma\lesssim_\varepsilon
		N\delta^{-3d-1/2-\varepsilon}\sigma^{2d-3}.
		\]
		Since $2d-3\le0$, the largest value is at $\sigma\asymp\sqrt\delta$, and equals $N\delta^{-2d-2-\varepsilon}$.
		
		For $3/2<d\le2$, $c=(2d-1)^{-1}$ and the same lemma reads
		\[
		\E_\sigma\lesssim_\varepsilon
		N\delta^{-(4d-1)-\varepsilon}\sigma^{4d-6}.
		\]
		Here $4d-6>0$, so the largest value is at $\sigma\asymp1$. The near-zero contribution in Lemma \ref{lem:nearnull} is $O(N\delta^{-2d-2})$, which is bounded by $N\delta^{-(4d-1)}$ for $d\ge3/2$. There are $O(\log(1/\delta))$ dyadic scales $\sigma$, and we  absorb this loss into $\delta^{-\varepsilon}$. 
	\end{proof}

	\section{$L^6$ Fourier decay for fractal measures}
	
	We now prove Theorem \ref{thm:fourierf}, which we recall below. The derivation from the energy bound is rather standard. One decomposes $\mu$ into {0,1}-valued densities $\mu_n$ that are locally constant at scale $\delta$. The different contributions are evaluated using the triangle inequality. The centers of the $\delta$-balls supporting $\mu_n$ satisfy KT non-concentration \eqref{kt}, albeit with constants that are not $O(1)$. 
	Random sampling is typically used in the literature to go from large  to $O(1)$ constants. The even-ness of 6 allows for the sixth moment of $\widehat{\mu_n}$ to be expressed in terms of the 3-energy of the centers. Rather than pursuing this somewhat technical approach, we follow a more elegant random sampling that could be of independent interest.

	\begin{theorem}\label{thm:fourier}
		Let $1<d\le2$, and let $\mu$ be a probability measure supported on $\C$ such that
		\begin{equation}\label{frostman}
			\mu(B(x,r))\le C r^d,\qquad 0<r\le 2.
		\end{equation}
		For every $\varepsilon>0$ and $R\ge2$,
		\begin{equation}\label{fourierbound}
			\int_{B(0,R)}|\widehat\mu(\xi)|^6\,d\xi
			\lesssim_{\varepsilon,C} R^{\varepsilon}
			\max\{R^{5-3d},R^{2-d}\}.
		\end{equation}
	\end{theorem}
	
	We first pass from point sets to measures at scale $\delta$. 
	
	\begin{lemma}\label{lem:measureenergy}
		Under the hypotheses of Theorem \ref{thm:fourier}, for $0<\delta\ll 1$ and any fixed $A>0$,
		\begin{equation}\label{measureenergy}
			\mu^{\otimes6}\!\left\{(p_i)_{i=1}^6:
			\left|p_1+p_2+p_3-p_4-p_5-p_6\right|\le A\delta\right\}
			\lesssim_{\varepsilon,C,A}
			\delta^{6d-\varepsilon}\max\{\delta^{-(3d+2)},\delta^{-(5d-1)}\}.
		\end{equation}
	\end{lemma}
	
	\begin{proof}

		Choose a maximal $\delta$-separated set $\{x_j\}$ in $\operatorname{supp}\mu$ and partition
		$\operatorname{supp}\mu$ into measurable  cells $D_j$ associated with this net.
		Each $D_j$ is contained in $B(x_j,\delta)$. Write $m_j=\mu(D_j)$.
		The points $x_j$ lie on $\C$, are $\delta$-separated, and satisfy
		$m_j\le C_1\delta^d$ by \eqref{frostman}.
		Set $a_j=m_j/(C_1\delta^d)\le1$.
		
		We follow a probabilistic approach. Independently for each $k=1,\ldots,6$ and each $j$, select the center $x_j$ with probability
		$a_j$. Denote the resulting six random sets by $S_k$, and put $S=\bigcup_{k=1}^6S_k$.
		
		We first prove the following: for every $D>0$ there is a constant $K_D$ such that, with probability at least
		$1-O_D(\delta^D)$,
		\begin{equation}\label{randomKT}
			\#(S\cap B(x,\rho))\le K_D\log(1/\delta)(\rho/\delta)^d,
			\qquad \delta\le\rho\le2.
		\end{equation}
		Here are the details. For each $j$ let $Y_j=\mathbf 1_{\{x_j\in S\}}$.
		The variables $Y_j$ are independent as $j$ varies, and
		\[
		\mathbb P(Y_j=1)=1-(1-a_j)^6\le 6a_j.
		\]
		Let $r_\ell=2^\ell\delta$, $0\le\ell\le L$, where $L$ is the smallest
		integer such that $r_L\ge2$. For each center $x_j$ of the original
		$\delta$-separated net (whether or not $x_j$ is selected into $S$), put
		\[
		X_{j,\ell}=\#(S\cap B(x_j,2r_\ell))
		=\sum_{i:\,x_i\in B(x_j,2r_\ell)}Y_i.
		\]
		Since the cells $D_i$ are disjoint and $D_i\subset B(x_i,\delta)$,
		while $\delta\le r_\ell$, we have
		\[
		\begin{aligned}
			\lambda_{j,\ell}:=\mathbb E X_{j,\ell}
			&\le \frac{6}{C_1\delta^d}
			\sum_{x_i\in B(x_j,2r_\ell)}\mu(D_i)\\
			&\le \frac{6\mu(B(x_j,3r_\ell))}{C_1\delta^d}
			\lesssim_C (r_\ell/\delta)^d.
		\end{aligned}
		\]
		When $3r_\ell>2$, the last inequality follows instead from
		$\mu(B)\le1$ and $r_\ell>2/3$.
		Put $t_\ell=K\log(1/\delta)(r_\ell/\delta)^d$, with $K$ to be chosen.
		The Chernoff bound for sums of independent Bernoulli variables gives
		\[
		\mathbb P(X_{j,\ell}\ge t_\ell)
		\le\left(\frac{e\lambda_{j,\ell}}{t_\ell}\right)^{t_\ell}
		\le 2^{-t_\ell}\le\delta^{K\log2},
		\]
		provided $K$ is sufficiently large in terms of $C,d$.
		There are $O(\log(1/\delta))$ choices of $\ell$ and
		$O(\delta^{-3})$ original centers $x_j$, by separation and the
		boundedness of $\operatorname{supp}\mu\subset\C$.
		Taking the union bound over these  pairs $(j,\ell)$,
		the probability that any estimate fails is at most
		\[
		O\bigl(\delta^{-3}\log(1/\delta)\,\delta^{K\log2}\bigr)
		=O_D(\delta^D),
		\]
		if $K$ is large enough depending on $D$.
		
		Fix $D>18$, and choose $K=K_D$ sufficiently
		large above. Let $G$ denote the event that $X_{j,\ell}\le t_\ell$
		simultaneously for every center $x_j$ of the original net and
		every dyadic radius $r_\ell$. The preceding union bound shows that
		\[
		\mathbb P(G^c)\lesssim_D\delta^D.
		\]
		For every outcome in $G$, we first verify that the
		non-concentration estimate \eqref{randomKT} holds for arbitrary balls.
		Take a ball $B(x,\rho)$ with $\delta\le\rho\le2$.
		If $S\cap B(x,\rho)$ is empty, there is
		nothing to prove. Otherwise, choose one of its selected points
		$x_j\in S\cap B(x,\rho)$. For every $z\in S\cap B(x,\rho)$,
		$|z-x_j|\le |z-x|+|x-x_j|<2\rho$, and hence
		\[
		S\cap B(x,\rho)\subset S\cap B(x_j,2\rho)
		\subset S\cap B(x_j,2r_\ell),
		\]
		where $\ell$ is chosen so that $\rho\le r_\ell<2\rho$.
		The simultaneous estimates for the original centers therefore give
		\[
		\#(S\cap B(x,\rho))\le X_{j,\ell}\le t_\ell
		\le 2^d K\log(1/\delta)(\rho/\delta)^d.
		\]
		Absorbing $2^d$ into $K_D$ proves
		\eqref{randomKT} simultaneously for all centers and radii on $G$.
		
		Let $\Gamma$ consist of (deterministic) ordered six-tuples of cell indices whose representatives satisfy
		$|x_{j_1}+x_{j_2}+x_{j_3}-x_{j_4}-x_{j_5}-x_{j_6}|\le(A+6)\delta$.
		We first relate this discrete sum to the measure on the left of
		\eqref{measureenergy}. Write
		\[
		F_{A,\delta}:=\Big\{(p_1,\ldots,p_6)\in\C^6:
		\Big|\sum_{k=1}^3p_k-\sum_{k=4}^6p_k\Big|\le A\delta\Big\}.
		\]
		If $(p_1,\ldots,p_6)\in F_{A,\delta}$ and $p_k\in D_{j_k}$,
		then $|p_k-x_{j_k}|\le\delta$ for each $k$, and hence
		\[
		\begin{aligned}
			\Big|\sum_{k=1}^3x_{j_k}-\sum_{k=4}^6x_{j_k}\Big|
			&\le \Big|\sum_{k=1}^3p_k-\sum_{k=4}^6p_k\Big|
			+\sum_{k=1}^6|p_k-x_{j_k}|\\
			&\le(A+6)\delta.
		\end{aligned}
		\]
		Thus $(j_1,\ldots,j_6)\in\Gamma$. Since the cells form a
		measurable partition and $\mu^{\otimes6}$ is a product measure,
		\begin{equation}\label{measuretodiscrete}
			\begin{aligned}
				\mu^{\otimes6}(F_{A,\delta})
				&=\sum_{j_1,\ldots,j_6}
				\mu^{\otimes6}\bigl(F_{A,\delta}\cap
				(D_{j_1}\times\cdots\times D_{j_6})\bigr)\\
				&\le\sum_{(j_1,\ldots,j_6)\in\Gamma}
				\mu(D_{j_1})\cdots\mu(D_{j_6})
				=\sum_{\Gamma}m_{j_1}\cdots m_{j_6}.
			\end{aligned}
		\end{equation}

		Define the random count
		\[
		Z:=\#\{(j_1,\ldots,j_6)\in\Gamma:
		x_{j_k}\in S_k\ (1\le k\le6)\}.
		\]
		Independence of the six selections, even when some indices coincide, gives
		\[
		\mathbb E Z=\sum_{\Gamma}a_{j_1}\cdots a_{j_6}.
		\]
		As $m_j=C_1\delta^d a_j$, \eqref{measuretodiscrete} now yields
		\begin{equation}\label{measuretoexpectation}
			\mu^{\otimes6}(F_{A,\delta})
			\le\sum_{\Gamma}m_{j_1}\cdots m_{j_6}
			=(C_1\delta^d)^6\mathbb E Z.
		\end{equation}
		For each outcome, $Z$ is bounded by the energy $\E_{3,\delta}(S)$,
		where we take the fixed tolerance constant $C_*=A+6$ in the definition
		of that energy; the underlying separation scale remains $\delta$.
		Split the expectation using the  good event $G$:
		\[
		\mathbb E Z=\mathbb E(Z\mathbf 1_G)+\mathbb E(Z\mathbf 1_{G^c}).
		\]
		On $G$, the Katz--Tao estimate \eqref{randomKT} holds for $S$.
		Consequently, for each individual outcome $\omega\in G$,
		Theorem~\ref{thm:main}, applied to the $\delta$-separated set
		$S(\omega)\subset\C$ with tolerance constant $C_*=A+6$, gives the
		uniform pointwise upper bound
		\[
		Z(\omega)\le\E_{3,\delta}(S(\omega))
		\lesssim_{\varepsilon,C,D,A}\delta^{-\varepsilon}
		\max\{\delta^{-(3d+2)},\delta^{-(5d-1)}\}.
		\]
		Thus
		\[
		\begin{aligned}
			\mathbb E(Z\mathbf 1_G)
			&\lesssim_{\varepsilon,C,D,A}\delta^{-\varepsilon}
			\max\{\delta^{-(3d+2)},\delta^{-(5d-1)}\}\,\mathbb P(G)\\
			&\lesssim_{\varepsilon,C,D,A}\delta^{-\varepsilon}
			\max\{\delta^{-(3d+2)},\delta^{-(5d-1)}\}.
		\end{aligned}
		\]
		On $G^c$ we use only the deterministic estimate
		$Z\le\#\Gamma\lesssim\delta^{-18}$, as there are $O(\delta^{-3})$
		original net points. Thus
		\[
		\mathbb E(Z\mathbf 1_{G^c})
		\lesssim\delta^{-18}\mathbb P(G^c)
		\lesssim_D\delta^{D-18}.
		\]
		Our fixed choice $D>18$ makes this term bounded (indeed vanishing)
		as $\delta\to0$. Together with the bound on $G$, this yields
		\[
		\mathbb E Z
		\lesssim_{\varepsilon,C,A}\delta^{-\varepsilon}
		\max\{\delta^{-(3d+2)},\delta^{-(5d-1)}\}.
		\]
		Multiplying by $(C_1\delta^d)^6$ in \eqref{measuretoexpectation}
		proves \eqref{measureenergy}.
		
	\end{proof}
	
	\begin{proof}[Proof of Theorem \ref{thm:fourier}]
		Choose a nonnegative, smooth, compactly supported function $g$ on $\R^3$, of
		sufficiently small support and with $\int g=1$. Put
		$\phi=g*\widetilde g$, where $\widetilde g(x)=g(-x)$.
		Then $\phi\ge0$, $\widehat\phi=|\widehat g|^2\ge c>0$ on $B(0,1)$,
		and $\phi$ is supported in a fixed ball $B(0,A)$.
		With $\nu=\mu*\mu*\mu$ and $\widetilde\nu(E)=\nu(-E)$, Fourier inversion gives
		\begin{align*}
			\int_{B(0,R)}|\widehat\mu(\xi)|^6\,d\xi
			&\lesssim \int_{\R^3}|\widehat\mu(\xi)|^6
			\widehat\phi(\xi/R)\,d\xi\\
			&=R^3\int_{\R^3}\phi(Rx)\,d(\nu*\widetilde\nu)(x)\\
			&\lesssim R^3\,\mu^{\otimes6}\!\left\{(p_i):
			\left|\sum_{i=1}^3p_i-\sum_{i=4}^6p_i\right|\le A/R\right\}.
		\end{align*}
		Apply Lemma \ref{lem:measureenergy} with $\delta=R^{-1}$. The last expression is
		\[
		\lesssim_\varepsilon
		R^{3-6d+\varepsilon}\max\{R^{3d+2},R^{5d-1}\}
		=R^\varepsilon\max\{R^{5-3d},R^{2-d}\},
		\]
		as required.
	\end{proof}
	\begin{remark}
		\label{djhfdhufhrgtgu9thu90}
		We close by presenting the sharpness of the exponents in Theorem \ref{thm:fourier}.
		For each $R\geq 2$, set
		$\delta=R^{-1}$ and choose a maximal-size $(\delta,d)$-KT set
		$P$ that maximizes the energy
		bound. Let $\mu$ be normalized surface measure on the union
		of small $\delta$-neighborhoods of $P$ within the cone. The KT
		condition makes $\mu$ a $d$-Frostman measure with constant
		independent of $R$. A standard positive-kernel argument
		then converts the discrete energy lower bound into the corresponding
		Fourier lower bound. Thus the exponents are sharp uniformly over
		$d$-Frostman measures, allowing the measure to depend on $R$.
	\end{remark}

	\paragraph{Use of AI.} The mathematical results in this paper are mostly due to the author. ChatGPT (GPT-6, OpenAI) was used to sharpen the exposition, to generate the map $V$ and the alternative proof to Theorem \ref{thm:fourier} presented here. 
	
	\bigskip
	\noindent Department of Mathematics, Indiana University, 831 East 3rd St., Bloomington, IN 47405, USA.
	
	\noindent \textit{Email address:} \texttt{demeterc@iu.edu}
\end{document}